\documentclass[12pt]{article}
\usepackage{amsmath}
\usepackage{amsfonts}
\usepackage{amssymb}
\usepackage{amsthm}
\usepackage{stackrel}

\usepackage{color}

\usepackage{figsize}

\usepackage{graphics}

\newcommand{\mc}{\mathcal }
\newtheorem{ex}{Example}[section]
\newtheorem{as}{Assumption}

\newtheorem{thm}{Theorem}[section]

\newcommand{\go}[1]{\mathfrak{#1}}

\newcommand{\R}{{\rm I}\kern-0.18em{\rm R}}
\newcommand{\1}{{\rm 1}\kern-0.25em{\rm I}}
\newcommand{\E}{{\rm I}\kern-0.18em{\rm E}}
\newcommand{\p}{{\rm I}\kern-0.18em{\rm P}}
\makeatletter

\def\@fnsymbol#1{\ensuremath{\ifcase#1\or a\or b\or c\or d\or \e\or f\or *\dagger 	\or \ddagger\ddagger \else\@ctrerr\fi}}

\title{A method of induction the distances with Hilbert structure}
\author{Vesna Gotovac\footnote{Department of Mathematics, Faculty of Science, University of Split, 21000 Split, Croatia}, Kate\v{r}ina Helisova \footnote{Department of Mathematics, Faculty of Electrical Engineering, Czech Technical University in Prague, Technick\'{a} 2, 166 27 Prague 6, Czech Republic},\\ Lev B. Klebanov\footnote{Department of Probability and Mathematical Statistics, Charles University, Sokolovska 83, 18675 Prague, Czech Republic, e-mail: levbkl@gmail.com} and Irina V. Volchenkova\footnotemark[2]\, \footnote{The order of authors is alphabetical and has no other sense}}

\date{}
\begin{document}

\maketitle

\begin{abstract} 
A method of induction the distances with Hilbert structure is proposed. Some properties of the method are studied. Typical examples of corresponding metric spaces are discussed.

\end{abstract}

\section{Introduction}\label{sec1} 
\setcounter{equation}{0}

Let $\{\mc X, D\}$ be a metric space. We say $D$ is a Hilbert-type distance if and only if there is an isometry from $\{\mc X, D\}$ on a subset of a Hilbert space. It is known (see \cite{Sch}) this property is equivalent to negative definiteness of $D^2$. Namely, a real function $\mc L$ from $\mc X^2$ such that $\mc L(x_1,x_2)=\mc L(x_2,x_1)$ is called negative definite kernel if for arbitrary positive integer $n$ and real numbers $c_1, \ldots , c_n$ satisfying to the condition $\sum_{j=1}^{n} c_j =0$ we have
\begin{equation}\label{eq0}
\sum_{i=1}^{n}\sum_{j=1}^{n} \mc L(x_i,x_j)c_ic_j \leq 0.
\end{equation}
That is, to proof that $D$ is a Hilbert-type distance it is sufficient to verify $\mc L =D^2$ satisfies the relation (\ref{eq0}). In the paper we provide a method of constructing Hilbert-type distance on a set $\mc X$ by using corresponding distance on image of $\mc X$ under a family of functions.

\section{The method of defining Hilbert-type distances}\label{sec2} 
\setcounter{equation}{0}

Let $\mc Z$ be a metric space with a distance $D$ on it. It is well-known that $\mc Z$ is isometric to a set of a Hilbert space if and only if $D^2(u,v)$ $(u,v \in Z)$ is negative definite kernel of $\mc Z^2$. Further on we suppose that $D$ possesses this property. 
Let $\mc X$ be an abstract set, and let $f_y(.), \; y\in {\mc Y}$ be a family of functions defined on $\mc X$ and taking values in $\mc Z$. Suppose that $\Xi$ is a probability measure on $\mc Y$. Define
\begin{equation}\label{eq1} 
\rho(x_1,x_2)= \Bigl(\int_{\mc Y} D^2(f_y(x_1),f_y(x_2))d\,\Xi(y)\Bigr)^{1/2}.
\end{equation}
Our goal is to show that under some assumptions $\rho$ is a distance on $\mc X$ such that $\rho^2$ is a negative definite kernel.

It is clear that for $x_1,x_2 \in {\mc X}$
\begin{enumerate} 
\item[A1.] $\rho(x_1,x_2) \geq 0$;
\item[A2.] $\rho(x_1,x_2) = 0$ if and only if 
$f_y(x_1)=f_y(x_2)$ for $\Xi$-almost all $y \in \mc Y$; 
\item[A3.] $\rho(x_1,x_2)=\rho(x_2,x_1)$.
\end{enumerate}
Take arbitrary positive integer $n$ and real numbers $c_1, \ldots , c_n$ satisfying to the condition $\sum_{j=1}^{n} c_j =0$. For arbitrary $x_1, \ldots , x_n \in \mc X$ we have
\begin{equation}\label{eq2}
\sum_{i=1}^{n}\sum_{j=1}^{n}\rho^2(x_i,x_j)c_ic_j= \int_{\mc Y}\Bigl( \sum_{i=1}^{n}\sum_{j=1}^{n} D^2(f_y(x_i),f_y(x_j))c_ic_j \Bigr) d\, \Xi(y) \leq 0
\end{equation}
in view of negative definiteness of the kernel $D^2$. Therefore, $\rho^2$ is negative definite kernel on $\mc X$. From this fact it follows that $\rho$ satisfies the triangle inequality (see, for example, \cite{Kl}):
\begin{enumerate} 
\item[A4.] $\rho(x_1,x_2) \leq \rho(x_1,x_3)+ \rho(x_3,x_2)$ for all $x_1,x_2,x_3 \in \mc X$.
\end{enumerate}

\begin{thm}\label{th1} 
Let $\mc Z$ be a metric space with a distance $D$ on it. Suppose that $D^2(u,v)$ $(u,v \in Z)$ is negative definite kernel of $\mc Z^2$. Let $\mc X$ be an abstract set, and let $f_y(.), \; y\in {\mc Y}$ be a family of functions defined on $\mc X$ and taking values in $\mc Z$. Suppose that $\Xi$ is a probability measure on $\mc Y$ such that 
\begin{equation}\label{eq3} 
f_y(x_1) =f_y(x_2)\; \text{for all} \; y\in \tt{supp} (\Xi) \; \text{implies} \; x_1=x_2,
\end{equation}
where, as usual, $\tt{supp} (\Xi)$ is support of the measure $\Xi$. Then the function $\rho$ defined by (\ref{eq1}) is a distance  on $\mc X$. Metric space $(\mc X, \rho)$ is isometric to a subset of a Hilbert space.
\end{thm}
\begin{proof}
Properties $A1. - A4.$ together with (\ref{eq3}) show that $\rho$ is a distance on $\mc X$. Because $\rho^2$ is a negative definite kernel the conclusion of the Theorem follows from I.J. Schoenberg's Theorem (\cite{Sch}, see also \cite{Kl}).
\end{proof}

\begin{ex}\label{ex1} 
Let $\mc X$ be a subset of a vector space $\mc H$ and $\mc Y$ is a subspace of algebraic conjugate $\mc X^{\prime}$. Suppose that on $\mc X^{\prime}$ with a $\sigma$-field of its subsets there exists a measure $\Xi$ such that
\[ \langle x^{\prime},x_1\rangle  = \langle x^{\prime},x_2\rangle \; \text{for all}\; x^{\prime}\; \text{implies}\;x_1 = x_2. \]
Then there exists a Hilbert-type distance on $\mc X$. 
\end{ex}
\begin{proof}
It is sufficient to apply Theorem \ref{th1} to
$\mc Y = \mc X^{\prime}$ and $D(u,v)=|u-v|$ for $u,v \in \mc Z =\R^1$.
\end{proof}

Let us note that the distance
\[ \rho(x_1,x_2) =\Bigl(\int_{\mc Y} (\langle x^{\prime},x_1\rangle-\langle x^{\prime},x_2\rangle)^2 d\,\Xi(x^{\prime})\Bigr)^{1/2} \]
on $\mc X$ induces a norm on $\mc X$. Namely,
\[ \|x\| = \rho(x,0), \; x\in \mc X. \]
Corresponding inner product may be calculated as
\[ (x_1,x_2)= \int_{\mc Y} \langle x^{\prime},x_1\rangle\cdot \langle x^{\prime},x_2\rangle d\,\Xi(x^{\prime}). \]
Basing on this, we can say that $\mc X$ is isometric to linear subspace of a Hilbert space. However, we cannot state that $\mc X$ is complete.

The conditions of Example \ref{ex1} are close to necessary. Really, let $\mc X$ be a subset of a separable Hilbert space $\go H$. We may take as $\mc Y$ the dual space of all continuous linear functionals on $\go H$. Of course, in this situation there exists corresponding measure $\Xi$ possessing desirable properties. Namely, from the proof of the result of \cite{Ito} it follows that we can choose $\Xi$ as a Gaussian measure with
$\tt supp (\Xi)={\go H}$.

The facts given by Example \ref{ex1} and after it show that the condition of existence of inner product on a linear space may be changed by the condition of existence of suitable measure $\Xi$ on reach enough subset of the dual space.

\section{The method of defining $L^m$-type distances}\label{sec3} 
\setcounter{equation}{0}

Let $\mc X$ be an abstract set, and $m$ be an even integer greater than $1$. Assume that $\mc L(x_1, \ldots , x_m)$ is a real continuous function on $\mc X^m$ symmetric with respect to permutations of its arguments. We say that $\mc L$ is an {\it $m$-negative definite kernel} (see \cite{ZKK}) if for any integer $n \geq 1$, any collection of points $x_1, \ldots ,x_n \in \mc X$, and any collection of real numbers $h_1, \ldots , h_n$ satisfying the condition $\sum_{j=1}^{n}h_j =0$, the following inequality holds:
\begin{equation}\label{eq4}
(-1)^{m/2}\sum_{i_1=1}^{n} \ldots \sum_{i_m=1}^{n} \mc L(x_{i_1},\ldots ,x_{i_m})h_{i_1}\cdots h_{i_m} \geq 0.
\end{equation}
For the case $m=2$ we have the case of negative definite kernel. If the equality in (\ref{eq4}) implies that $h_1=\ldots =h_n=0$, then we call $\mc L$ {\it strictly $m$-negative definite kernel}. Equivalent form of (\ref{eq4}) is
\begin{equation}\label{eq5}
(-1)^{m/2}\int_{\mc X} \ldots \int_{\mc X} \mc L (x_1, \ldots ,x_m)h(x_1)\cdots h(x_m) dQ(x_1) \ldots dQ(x_m)\geq 0
\end{equation}
for any measure $Q$ and any integrable function $h(x)$ such that
\begin{equation}\label{eq6}
 \int_{\mc X} h(x) dQ(x) =0.
\end{equation}
We call $\mc L$ {\it strongly $m$-negative definite kernel} if the equality in (\ref{eq6}) holds for $h=0$ $Q$-almost everywhere only.

Let $\mc X$ be an abstract set, and let $f_y(.), \; y\in {\mc Y}$ be a family of functions defined on $\mc X$ and taking values in a set $\mc Z$. Suppose that $\mc L$ is $m$-negative definite kernel on $\mc Z^m$ and $\Xi$ is a probability measure on $\mc Y$. Define
\begin{equation}\label{eq7}
\go R_m (x_1, \ldots ,x_m)=\int_{\mc Y} \mc L (f_y(x_1),\ldots ,f_y(x_m))d\,\Xi(y).
\end{equation}
It is easy to see that $\go R_m$ is $m$-negative definite kernel on $\mc X^m$.

\begin{as} 
If 
\[(-1)^{m/2}\sum_{i_1=1}^{n} \ldots \sum_{i_m=1}^{n} \mc L(f_y(x_{i_1}),\ldots ,f_y(x_{i_m}))h_{i_1}\cdots h_{i_m} =0 \] \[\text{for}\; Q-\;\text{almost all} \; y\in \mc Y \] 
implies that
\[(-1)^{m/2}\sum_{i_1=1}^{n} \ldots \sum_{i_m=1}^{n} \mc L(x_{i_1},\ldots ,x_{i_m})h_{i_1}\cdots h_{i_m} =0 \]
then strictly $m$-negativeness of $\mc L$ implies strictly $m$-negativeness of $\go R$.
\end{as}

Similarly statements with integrals instead of sums are true for strong negative definiteness. We omit precise formulation.

Suppose, as before, that $m$ is an even integer greater than $1$. Let $D_m$ is a distance on $\mc Z$. From the results of \cite{ZKK} (see also \cite{Kl}) it follows that $(\mc Z,D_m)$ is isometric to a subset of $L^{m}$ space if and only if 
\begin{equation}\label{eq8}
D_m(u,v) = \Bigl( (-1)^{m/2}\mc L (u-v,\ldots , u-v)\Bigr)^{1/m},\quad u,v \in \mc Z,
\end{equation}
where $\mc L (u_1,\ldots ,u_m)$ is a strictly negative definite kernel on $\mc Z^m$. Therefore, under Assumption 1 and if $D_m(u,v)$ has the form (\ref{eq8}) then
\begin{equation}\label{eq9}
\rho_m(s,t) = \Bigl( \go R_m(s-t,\ldots ,s-t)\Bigr)^{1/m}, \quad s,t \in \mc X
\end{equation}
is a distance on $\mc X$, where $L$ is strictly $m$-negative definite kernel used in (\ref{eq8}). The set $\mc X$ with the distance $\rho_m$ is isometric to a subset of $L^m$ space.

\section{Acknowledgment}
The work was partially supported by Grant GACR 16-03708S.

\end{document}